\documentclass[a4paper,12pt,reqno]{amsart}
\usepackage{a4wide}
\usepackage{amsmath}
\usepackage{amssymb}
\usepackage{amsthm}
\usepackage{latexsym}
\usepackage{graphicx}
\usepackage[english]{babel}

\usepackage{pgfplots}
\pgfplotsset{compat=1.15}
\usepackage{mathrsfs}
\usetikzlibrary{arrows}

\usepackage{lineno}

\newtheorem{obs} [subsection]{Remark}

\newtheorem{conj}[subsection]{Conjecture}
\newtheorem{teor}[subsection]{Theorem}
\newtheorem{lema}[subsection]{Lemma}
\newtheorem{cor} [subsection]{Corollary}

\def\qdepth{\operatorname{hdepth}}

\begin{document}
\selectlanguage{english}
\frenchspacing

\numberwithin{equation}{section}

\title{On the Hilbert depth of quadratic and cubic functions}
\author[Silviu B\u al\u anescu and Mircea Cimpoea\c s]{Silviu B\u al\u anescu$^1$ and Mircea Cimpoea\c s$^2$}
\date{}



\footnotetext[1]{ \emph{Silviu B\u al\u anescu}, University Politehnica of Bucharest, Faculty of
Applied Sciences, 
Bucharest, 060042, E-mail: silviu.balanescu@stud.fsa.upb.ro}
\footnotetext[2]{ \emph{Mircea Cimpoea\c s}, University Politehnica of Bucharest, Faculty of
Applied Sciences, 
Bucharest, 060042, Romania and Simion Stoilow Institute of Mathematics, Research unit 5, P.O.Box 1-764,
Bucharest 014700, Romania, E-mail: mircea.cimpoeas@upb.ro,\;mircea.cimpoeas@imar.ro}

\begin{abstract}
Given a numerical function $h:\mathbb Z_{\geq 0}\to\mathbb Z_{\geq 0}$ with $h(0)>0$, the Hilbert depth of $h$
is $\qdepth(h)=\max\{d\;:\;\sum\limits_{j=0}^k (-1)^{k-j}\binom{d-j}{k-j}h(j)\geq 0\text{ for all }k\leq d\}.$

In this note, we study the Hilbert depth of the functions $h_2(j)=aj^2+bj+e$, $j\geq 0$, and 
$h_3(j)=aj^3+bj^2+cj+e$, $j\geq 0$, where $a,b,c,e$ are some integers with $a,e>0$. 
We prove that if $b<0$ and $b^2\leq 4ae$ then $\qdepth(h_2)\leq 11$, and, if $b<0$
and $b^2>4ae$ then $\qdepth(h_2)\leq 13$. Also, we show that if $b<0$ and $b^2\leq 3ac$ then $\qdepth(h_3)\leq 67$.
\end{abstract}

\maketitle

\section{Introduction}

Let $S=K[x_1,\ldots,x_n]$ be the ring of polynomials in $n$ variables over a field $K$.
The Hilbert depth of a finitely graded $S$-module $M$ is the maximal depth of a finitely graded $S$-module $N$ with the same
Hilbert series as $M$; see \cite{uli} for further details.

In \cite{lucrare7} we introduced the notion of (arithmetic) Hilbert depth of a numerical function with nonnegative integer values, as a
generalization of the Hilbert depth of quotient $M=J/I$ of two squarefree monomial ideals $0\subset I\subsetneq J\subset S$, 
which we previously studied in \cite{lucrare2}. 

Here we will use a slightly modified version of it. Let
$$\mathcal H_0:=\{h:\mathbb Z_{\geq 0}\to \mathbb Z_{\geq 0}\;:\;h(0)>0\}.$$
Let $h\in \mathcal H_0$. The \emph{Hilbert depth} of $h$ is 
$$\qdepth(h)=\max\{d\;:\;\beta_k^d(h):=\sum_{j=0}^k (-1)^{k-j}\binom{d-j}{k-j}h(j)\geq 0\text{ for all }0\leq k\leq d\}.$$
According to \cite[Proposition 1.5]{lucrare7} we have that 
\begin{equation}\label{11}
\qdepth(h)\leq c(h):=\left\lfloor \frac{h(1)}{h(0)} \right\rfloor.
\end{equation}
We recall the following result:

\begin{teor}\label{t1}(see \cite[Theorem 1.15]{lucrare7})
Let $h\in\mathcal H_0$, $h(j)=a_nj^n+\cdots+a_1j+a_0,\; j\geq 0$, such that $a_i\geq 0$ for all $1\leq i\leq n$
and $a_0>0$. Then $$\qdepth(h)\leq 2^{n+1}.$$
\end{teor}

Note that the upper bound given in Theorem \ref{t1} does not depend on the coefficients $a_i$'s, but only on $n$.
However, if we allow some $a_i$'s to be negative, then 
the situation become more complex. We propose the following conjecture:

\begin{conj}\label{conj}
For any $n\geq 1$, there exists a constant $C(n)>0$ such that
$$\qdepth(h)\leq C(n)\text{ for any }h\in \mathcal H_0\text{ with }h(j)=a_nj^n+\cdots+a_1n+a_0,\;
a_i\in\mathbb Z\text{ and }a_0>0.$$
\end{conj}

Note that, the case $n=1$ follows immediately from Theorem \ref{t1}.
The aim of this note is to tackle this problem for $n=2$ and $n=3$, that is, to 
study the Hilbert depth of the functions $h_2,h_3\in\mathcal H_0$, defined by
$$h_2(j)=aj^2+bj+e,\;j\geq 0,\text{ and }h_3(j)=aj^3+bj^2+cj+e,\;j\geq 0,\text{ where }a,e>0.$$
Our method consist in studying the functions
\begin{align*}
& f(x)=\frac{1}{e}\beta_2^x(h),\text{ where }\beta_2^x(h)=\binom{x}{2}h(0)-(x-1)h(1)+h(2)\text{ and } \\
& g(x)=\frac{1}{e}\beta_3^x(h),\text{ where }\beta_2^x(h)=-\binom{x}{3}h(0)+\binom{x-1}{2}h(1)-(x-2)h(2)+h(3),
\end{align*}
and $h=h_2$ or $h=h_3$. In Theorem \ref{teo1} we prove that 
$$\qdepth(h_2)\leq 11,\text{ if }b<0\text{ and }b^2\leq 4ae.$$
Also, in Theorem \ref{teo11} we show that 
$$\qdepth(h_2)\leq 13,\text{ if }b<0\text{ and }b^2>4ae.$$
Since, according to Theorem \ref{t1}, $\qdepth(h_2)\leq 8$ for $b\geq 0$,
it follows that Conjecture \ref{conj} holds for $n=2$; see Corollary \ref{cor11}.

In the case of $h_3$, we are able to give only partial answers to Conjecture \ref{conj}.
More precisely, in Theorem \ref{teo2}, we show that
$$\qdepth(h_3)\leq 60,\text{ if }b<0\text{ and }b^2\leq 3ac.$$
Also, in Remark \ref{ultima} we describe a method which gives upper bounds for $\qdepth(h_3)$, when the 
coefficients of $h_3$ satisfy certain inequalities.

\section{The quadratic function}

Let $h_2\in \mathcal H_0$ be defined by
$$h_2(j):=aj^2+bj+e\text{ for all }j\geq 0,$$ 
where $a,b,e\in \mathbb Z$, $a\neq 0$ and $h_2(0)=e>0$. 

Since, by hypothesis, $h_2(j)\geq 0$ for all $j\geq 0$, it follows that $a>0$.
We denote $\alpha:=\frac{a}{e}$ and $\beta:=\frac{b}{e}$. By \eqref{11}, we have that
\begin{equation}\label{ch}
\qdepth(h_2)\leq c(h_2)=\left\lfloor \frac{h_2(1)}{h_2(0)} \right \rfloor = \lfloor \alpha + \beta \rfloor + 1.
\end{equation}
We denote $\Delta=b^2-4ae$. Of course, if $\Delta<0$ then $h_2(j)>0$ for all $k\geq 0$.

\begin{lema}\label{l1}
With the above notations:
\begin{enumerate}
\item[(1)] If $a+b\geq 0$ then $h_2(j)>0$ for all $j\geq 0$. Also, $\qdepth(h_2)\geq 1$.
\item[(2)] If $a+b\geq 0$ and $\Delta=0$ then there exists an integer $k\geq 2e$ such that
           $ae=k^2$ and $b=\pm 2k$.
\item[(3)] If $a+b<0$ then $\qdepth(h_2)=0$.
\item[(4)] If $a+b<0$ and $\Delta = 0$ then $h_2(j)\geq 0$ for all $j\geq 0$.
           Moreover, $h_2(j)=0$ if and only if $j=-\frac{b}{2a}\in\mathbb Z$.
\item[(5)] If $a+b<0$ and $\Delta > 0$ then $h_2(j)\geq 0$ for all $j\geq 0$ 
           if and only if $\Delta\leq a^2$ and 
           $$ I:=\left[ \frac{-b-\sqrt{\Delta}}{2a},\frac{-b+\sqrt{\Delta}}{2a} \right] \subseteq 
					    [\ell,\ell+1],$$
						where $\ell:=\left\lfloor \frac{-b-\sqrt{\Delta}}{2a} \right\rfloor$.
\end{enumerate}
\end{lema}

\begin{proof}
(1) If $b\geq 0$ then the conclusion is clear. Assume that $b<0$ and $a+b\geq 0$.
    If $\Delta \geq 0$ then $h_2(x)\leq 0$ if and only if 
		$$x\in I:=\left[ \frac{-b-\sqrt{\Delta}}{2a},\frac{-b+\sqrt{\Delta}}{2a} \right].$$
		Since $\frac{-b}{2a}\leq \frac{1}{2}$ and $-b>\sqrt{\Delta}$ it follows that $I\subset (0,1)$.
		Hence, $h_2(j)>0$ for all $j\geq 0$. Also
		$$\beta_0^1(h_2)=h_2(0)>0,\;\beta_1^1(h_2)=h_2(1)-h_2(0)\geq 0,$$
		and thus $\qdepth(h_2)\geq 1$.
		
(2) From hypothesis, we have $b=-2\sqrt{ae}$. In particular, $ae=k^2$,
with $k>0$ and $b=\pm 2k$. On the other hand, since $a+b\geq 0$ it follows that $k\geq 2e$.

(3) Since $a+b<0$ it follows that $0<h_2(1)<h_2(0)$ and therefore, by \eqref{ch}, we have $c(h_2)=0$.
    Thus $\qdepth(h_2)=0$.
				
(4) If $\Delta=0$ then $h_2(x)=0$ if and only if $x=-\frac{b}{2a}$. Hence, we get the required conclusion.

(5) The condition $h_2(j)\geq 0$ for all $j\geq 0$ is fulfilled if and only if
    $$\mathring I \cap \mathbb Z = \emptyset,$$
		which implies that $I$ is a subset of an interval of the form $[\ell,\ell+1]$ with $\ell\geq 0$.
		Also, the length of $I$ is lower or equal to $\frac{1}{2}$ and thus $\sqrt{\Delta}\leq a$.
		
		(We denoted by $\mathring I$ the interior of $I$.)
\end{proof}


\begin{lema}\label{l2}
If $\alpha+\beta < 2$ then $\qdepth(h_2)=c(h_2)$. 
Also, if $\alpha+\beta\geq 2$ then $\qdepth(h_2)\geq 2$.
\end{lema}

\begin{proof}
Note that $h_2(1)\geq 0$ implies $a+b\geq -e$, and thus $\alpha+\beta\geq -1$.
If $a+b<0$ then $\alpha+\beta \in [-1,0)$ and thus, From Lemma \ref{l1} we have 
$\qdepth(h_2)=c(h_2)=0.$

If $a+b\geq 0$ then, again from Lemma \ref{l1}, we have that 
$c(h_2)\geq \qdepth(h_2)\geq 1.$
Hence, if $\alpha+\beta\in [0,1)$ then we have equalities above. 

Now, assume $\alpha+\beta \geq 1$, that is $a+b\geq e$. Then 
\begin{align*}
& \beta_0^2(h_2)=h_2(0)=e>0,\;\beta_1^2(h_2) = h_2(1)-2h_2(0) = a+b-e\geq 0\text{ and}\\
& \beta_2^2(h_2)=h_2(2)-h_2(1)+h_2(0)=3a+b+e\geq 2a+2e > 0.
\end{align*}
	and thus $\qdepth(h_2)\geq 2$. If $\alpha+\beta <2$ then $c(h_2) = 2$ and it
	follows that $\qdepth(h_2)=2$.
\end{proof}

We consider the quadratic function
\begin{equation}\label{fdx}
f(x):=\frac{1}{e}\beta_2^x(h_2)=\frac{1}{2}x^2-(\alpha+\beta+\frac{3}{2})x+(5\alpha+3\beta+2).
\end{equation}
Note that $f(x)$ has the discriminant $\Delta(\alpha,\beta):=(\alpha+\beta)^2-7\alpha-3\beta-\frac{7}{4}$
and, moreover, the solutions of the equation $f(x)=0$ are 
\begin{equation}\label{x12}
x_{1,2}=\alpha+\beta+\frac{3}{2} \pm \sqrt{\Delta(\alpha,\beta)}.
\end{equation}
Therefore, $f(d)<0$ if and only if $d\in (x_1,x_2)$.

\begin{lema}\label{l13}
Assume that $\alpha+\beta\geq 2$. If there exists an integer $3\leq d\leq c(h_2)+1$ 
such that $f(d)<0$ then $\qdepth(h_2)<d$.
\end{lema}

\begin{proof}
We have that 
\begin{align*}
& \beta_2^d(h_2)=\binom{d}{2}h_2(0)-\binom{d-1}{1}h_2(1)+\binom{d-2}{0}h_2(2) = \frac{d(d-1)}{2}e - (d-1)(a+b+e)+\\
& + (4a+2b+e)  = \frac{e}{2}d^2 - (a+b+\frac{3}{2}e)d + (5a+3b+2e).
\end{align*}
Since, by \eqref{fdx}, $\beta_2^d(h_2)=ef(d)$, we get the required conclusion.
\end{proof}


\begin{obs}\rm
Note that the equation $F(x,y)=(x+y)^2-7x-3y-\frac{7}{4}$ defines a parabola $\mathcal P\subset \mathbb R^2$ with the symmetry axis $x+y=\frac{5}{2}$,
the vertex $\left(-\frac{3}{4},\frac{13}{4}\right)$ and the focal point $\left(-\frac{1}{2},3\right)$. 
Moreover, the set $D=\{(x,y)\in\mathbb R^2\;:\;F(x,y)<0\}$ is the open connected subset of $\mathbb R^2$ bounded by $\mathcal P$
which contains the focal point of $\mathcal P$.
(We left the details for the reader, as an easy exercise.)

We consider the parabola $\mathcal P_1:G(x,y)=y^2-4x=0$ and the domain
$$D_1=\{(x,y)\in\mathbb R^2\;:\;G(x,y)> 0\}.$$
Note that $\Delta\geq 0$ if and only if $(\alpha,\beta)\in \overline{D_1}$.
We let $K:=\overline{D}\cap \overline{D_1}$. 

We claim that $K$ is compact. Indeed, this follows from the fact that
the intersection of $\mathcal P$ with $\mathcal P_1$ consist in 4 points:
$(\alpha_1,\beta_1)\approx (0.19,-0.87)$, $(\alpha_2,\beta_2)\approx (0.39,-1.26)$, $(\alpha_3,\beta_3)\approx (2.12,2.91)$ and 
$(\alpha_4,\beta_4)\approx(19.29,-8.78).$

Moreover, since the symmetry axis of $\mathcal P$ has the equation $x+y=\frac{5}{2}$, it follows that
$$\max_{(\alpha,\beta)\in K}(\alpha+\beta)=\alpha_4+\beta_4 \approx 10.51.$$

\begin{figure}[h]
	\centering	
\includegraphics[scale=0.17]{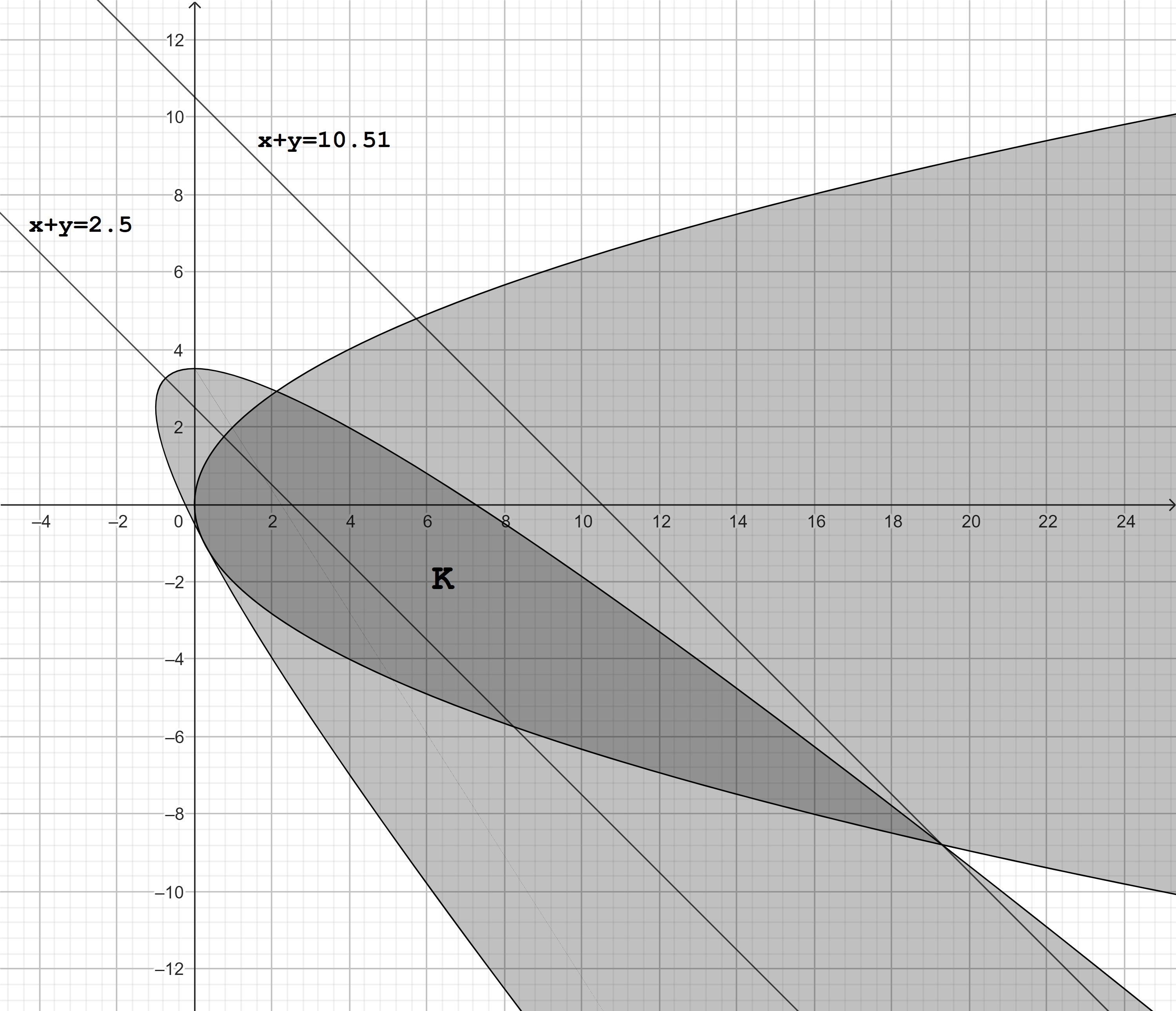}
\caption{The domain $K$.}
\end{figure}
\end{obs}

\begin{lema}\label{cheie}
If $\alpha+\beta\geq 11$ and $\Delta\leq 0$ then $\Delta(\alpha,\beta)>\frac{1}{4}$.
\end{lema}

\begin{proof}
Let $t:=\alpha+\beta$. We have that
\begin{equation}\label{tdab}
\Delta(\alpha,\beta) = (\alpha+\beta)^2 - 7\alpha-3\beta-\frac{7}{4} = t^2 -7t + 4\beta - \frac{7}{4}.
\end{equation}
On the other hand, we have that
$$\frac{\Delta}{e^2}=\beta^2 - 4\alpha=\beta^2 + 4\beta - 4t < 0,$$
therefore $\beta \in [-2-2\sqrt{1+t},-2+2\sqrt{1+t}]$. Since $\sqrt{1+t}<\frac{t}{3}$ for $t\geq 11$,
it follows that 
\begin{equation}\label{betaaa}
\beta \in \left(-2-\frac{2t}{3},-2+\frac{2t}{3}\right),\text{so }\beta>-2-\frac{2t}{3}.
\end{equation}
From \eqref{tdab} and \eqref{betaaa} it follows that
$$\Delta(\alpha,\beta) \geq t^2-7t+4(-2-\frac{2t}{3})-\frac{7}{4} = t^2-\frac{29}{3}t-\frac{39}{4}$$
Since $t\geq 11$, we obtain
$$\Delta(\alpha,\beta)-\frac{1}{4} = t^2-\frac{29}{3}t-10 = t(t-\frac{29}{3})-10\geq \frac{44}{3}-10 > 0,$$
as required.
\end{proof}

\begin{lema}\label{cheie2}
If $\Delta(\alpha,\beta)>\frac{1}{4}$ then $\qdepth(h_2)\leq 11$. 
\end{lema}

\begin{proof}
We consider the function 
$$\varphi(x)=x+\frac{3}{2}-\sqrt{x^2-7x+4\beta-\frac{7}{4}},\;x\in [11,\infty).$$
Let $t:=\alpha+\beta$. Since $\Delta>\frac{1}{4}$, from \eqref{tdab} it follows that
\begin{equation}\label{kukk1}
\varphi(t)=\alpha+\beta+\frac{3}{2}-\sqrt{\Delta(\alpha,\beta)} <\alpha +\beta + 1.
\end{equation}
On the other hand, by straightforward computations, we have that
\begin{equation}\label{kukk2}
\varphi(11)<12,\;\lim_{x\to\infty}\varphi(x)=5\text{ and }\varphi'(x)<0\text{ for all }x\in (11,\infty).
\end{equation}
On the other hand, from \eqref{11} we have that
$$\qdepth(h_2)\leq c(h_2) = \lfloor t \rfloor + 1 \leq t + 1.$$
Therefore, from \eqref{kukk1} and \eqref{kukk2} we get
$$\varphi(t)<d:=12\leq c(h_2)+1\leq t+2<\psi(t)=t+\frac{3}{2} + \sqrt{t^2-7t+4\beta-\frac{7}{4}}.$$ 
The conclusion follows from \eqref{fdx}, \eqref{x12} and Lemma \ref{l13}.
\end{proof}

\begin{teor}\label{teo1}
Let $h\in\mathcal H_0$, $h_2(j)=aj^2+bj+e$, $j\geq 0$.
\begin{enumerate}
\item[(1)] If $b\geq 0$ then $\qdepth(h_2)\leq 8$.
\item[(2)] If $b<0$ and $b^2\leq 4ae$, then $\qdepth(h_2)\leq 11$.
\end{enumerate}
\end{teor}

\begin{proof}
(1) It is a particular case of Theorem \ref{t1}.

(2) It follows from Lemma \ref{cheie} and Lemma \ref{cheie2}.
\end{proof}

\begin{obs}\label{remi}\rm
If we omit the condition $b^2\leq 4ae$ in Theorem \ref{teo1}(2), then
the sum $\alpha+\beta$ for which $\Delta(\alpha,\beta)<0$ can take values as large as we want.
For instance, if $k\geq 4$ is an integer and $h_2(j)= k^2j^2+ (k-k^2)j +1$, then we have $\alpha=k^2$ and $\beta=k-k^2$. Hence
$$\Delta=k^4-2k^3-3k^2>0\text{ and }\Delta(\alpha,\beta)=-4k^2-3k-\frac{7}{4}<0.$$
This shows that we cannot apply Lemma \ref{l13} in order to find an upper bound for $\qdepth(h_2)$, independent on coefficients.

However, if $k\geq 5$ then, by straightforward computations, we get
$$\beta_j^5(h_2)\geq 0\text{ for all }0\leq j\leq 5\text{ and }\beta_3^6(h_2)=-13k+5k-k^2<0.$$
Hence, $\qdepth(h_2)=5$. Thus, this is not a counterexample to Conjecture \ref{conj}. 
Moreover, this example gives a hint for what we need to do in the following.
\end{obs}

\begin{teor}\label{teo11}
Let $h\in\mathcal H_0$, $h_2(j)=aj^2+bj+e$, $j\geq 0$. 
If $b<0$ and $b^2>4ae$ then $$\qdepth(h_2)\leq 13.$$
\end{teor}

\begin{proof}
We consider the function
\begin{equation}\label{gdx}
g(x)=\frac{1}{e}\beta_3^x(h_2)=-\binom{x}{3}+\binom{x-1}{2}(\alpha+\beta+1)-(x-2)(4\alpha+2\beta+1)+(9\alpha+3\beta+1).
\end{equation}
By straightforward computations, from \eqref{gdx} we get
\begin{equation}\label{gdex}
g(x)=\alpha(\frac{1}{2}x^2-\frac{11}{2}x+18)+\beta(\frac{1}{2}x^2-\frac{7}{2}x+8)+(-\frac{1}{6}x^3+x^2-\frac{17}{6}x+4).
\end{equation}
Let $t=\alpha+\beta$. By replacing $\beta$ with $t-\alpha$ in \eqref{gdex} we get
\begin{equation}\label{gdeex}
g(x)=-2\alpha(x-5)+\frac{1}{2}t(x^2-7x+16)+(-\frac{1}{6}x^3+x^2-\frac{17}{6}x+4).
\end{equation}
In particular, from \eqref{gdeex} we get
\begin{equation}\label{g14}
g(14)=3(19t-6\alpha-99).
\end{equation}
Assume $\Delta(\alpha,\beta)\leq \frac{1}{4}$. From \eqref{tdab} it follows that
$$\Delta(\alpha,\beta)=t^2 -7t + 4\beta - \frac{7}{4} = t^2 - 3t - 4\alpha - \frac{7}{4} \leq \frac{1}{4},$$
which is equivalent to 
$$h(t)=t^2-3t-(4\alpha+2)\leq 0.$$
The discriminant of $h(t)$ is $\tilde{\Delta}=16\alpha+17$, hence $h(t)\leq 0$ implies
\begin{equation}\label{oho}
t\leq \frac{3+\sqrt{16\alpha+17}}{2}.
\end{equation} 
Let $s=\sqrt{16\alpha+17}$. We have $\alpha=\frac{s^2-17}{16}$. From \eqref{g14} and \eqref{oho} it follows that
$$\frac{1}{3}g(14) \leq   -\frac{1}{8}(3s^2 -76s + 513)$$
Since $76^2-4\cdot 3\cdot 513=-380<0$, it follows that $g(14)<0$. Therefore, from \eqref{gdx} we get
$\beta_3^{14}(h_2)<0$ and thus we are done. 
\end{proof}

\begin{cor}\label{cor11}
Let $h\in\mathcal H_0$, $h_2(j)=aj^2+bj+e$, $j\geq 0$. Then $\qdepth(h_2)\leq 13$. In particular,
Conjecture \ref{conj} holds for $n=2$.
\end{cor}

\begin{proof}
It follows from Theorem \ref{teo1} and Theorem \ref{teo11}.
\end{proof}

\section{The cubic function}

Let $h_3\in \mathcal H_0$ be defined by
$$h_3(j):=aj^3+bj^2+cj+e\text{ for all }j\geq 0,$$ 
where $a,b,c,e\in \mathbb Z$, $a\neq 0$ and $h(0)=e>0$. 

Since, by hypothesis, $h(j)\geq 0$ for all $j\geq 0$, it follows that $a>0$.
We denote $\alpha:=\frac{a}{e}$, $\beta:=\frac{b}{e}$ and $\gamma:=\frac{c}{e}$. By \eqref{11}, we have that
\begin{equation}\label{ch3}
\qdepth(h_3)\leq c(h_3)=\left\lfloor \frac{h(1)}{h(0)} \right \rfloor = \lfloor \alpha + \beta + \gamma \rfloor + 1.
\end{equation}

\begin{lema}\label{123}
If $\alpha+\beta+\gamma<2$ then $\qdepth(h_3)=c(h_3)$. Also, if 
$\alpha+\beta+\gamma\geq 2$ then $\qdepth(h_3)\geq 3$.
\end{lema}

\begin{proof}
The proof is the similar to the proof of Lemma \ref{l2}.
\end{proof}

The derivative of $h_3(x)$ is $h_3'(x)=3ax^2+2bx+c$ and it has the discriminant $$\Delta'=4b^2-12ac.$$
Therefore, if $b^2\leq 3ac$ then $h_3(x)$ is nondecreasing on $\mathbb R$ and, in particular,
the conditions $h_3(j)\geq h_3(0)>0$, for all $j\geq 0$, hold.
We have that
\begin{align*}
& \beta_2^d(h_3)=\binom{d}{2}h_3(0)-\binom{d-1}{1}h_3(1)+\binom{d-2}{0}h_3(2)=\frac{d(d-1)}{2}e - (d-1)(a+b+c+e) + \\ 
& + 8a+4b+2c+e = \frac{e}{2}d^2 - (a+b+c+\frac{3}{2}e)d + 9a+5b+3c+2e.
\end{align*}
We consider the quadratic function
\begin{equation}\label{fdx3}
f(x):=\frac{1}{e}\beta_2^x(h_3)=\frac{1}{2}x^2-(\alpha+\beta+\gamma+\frac{3}{2})x+(9\alpha+5\beta+3\gamma+2),
\end{equation}
which has the discriminant
\begin{equation}\label{delta3}
\Delta(\alpha,\beta,\gamma):=(\alpha+\beta+\gamma)^2-15\alpha-7\beta-3\gamma-\frac{7}{4}.
\end{equation}
With the above notations, we have:

\begin{lema}\label{l133}
Assume that $\alpha+\beta+\gamma\geq 2$. If there exists an integer $3\leq d\leq c(h_3)+1$ 
such that $f(d)<0$ then $\qdepth(h_3)<d$.
\end{lema}

\begin{proof}
The proof is similar to the proof of Lemma \ref{l13}.
\end{proof}

\begin{lema}\label{l134}
If $b<0$ and $\Delta'\leq 0$, i.e. $b^2\leq 3ac$, then 
$$15\alpha+7\beta+3\gamma \leq s_0(\alpha+\beta+\gamma),$$
where $s_0:=39+16\sqrt{3}\approx 66.71$.
\end{lema}

\begin{proof}
First, note that, since $a>0$ and $3ac\geq b^2>0$, we have $c>0$. Thus $\alpha>0$, $\beta<0$ and $\gamma>0$.
From the inequality of arithmetic and geometric means, it follows that 
\begin{equation}\label{kik1}
\beta \geq - \sqrt{3\alpha\gamma} \geq -\frac{\sqrt{3}}{2}(\alpha+\gamma).
\end{equation}
Let $s>15$. We have that
\begin{equation}\label{kik2}
15\alpha+7\beta+3\gamma < 15(\alpha+\gamma)+7\beta \leq s(\alpha+\beta+\gamma)
\text{ if and only if }(s-15)(\alpha+\gamma) \geq (s-7)(-\beta).
\end{equation}
From \eqref{kik1} and \eqref{kik2}, it follows that
$$15\alpha+7\beta+3\gamma\leq k(\alpha+\beta+\gamma),\text{ if }s-15\geq \frac{\sqrt 3}{2}(s-7),$$
which is equivalent to $s\geq s_0$, as required. 
\end{proof}

\begin{lema}\label{cheita}
Let $t_0:=\frac{s_0+\sqrt{s_0^2+8}}{2}\approx 66.74$.
If $b<0$ and $b^2\leq 3ac$ then
$$\Delta(\alpha,\beta,\gamma)>\frac{1}{4},\text{ for }\alpha+\beta+\gamma > t_0.$$
In particular, the inequality holds for $\alpha+\beta+\gamma\geq 67$.
\end{lema}

\begin{proof}
Let $t:=\alpha+\beta+\gamma$. From \eqref{delta3} and Lemma \ref{l134} it follows that
$$\Delta(\alpha,\beta,\gamma)\leq t^2-s_0t-\frac{7}{4}.$$
Therefore, in order to complete the proof it is enough to show that $t_0(t_0-s_0)>2$,
which is obviously true from the hypothesis.
\end{proof}

\begin{teor}\label{teo2}
Let $h_3\in\mathcal H_0$, $h_3(j)=aj^3+bj^3+cj+e$, $j\geq 0$.
\begin{enumerate}
\item[(1)] If $b,c\geq 0$ then $\qdepth(h_3)\leq 16$.
\item[(2)] If $b<0$ and $b^2\leq 3ac$ then $\qdepth(h_3)\leq 67$.
\end{enumerate}
\end{teor}

\begin{proof}
(1) It is a particular case of Theorem \ref{t1}.

(2) The proof is similar to the proof of Theorem \ref{teo1}, using Lemma \ref{cheita}.
\end{proof}

\begin{obs}\label{ultima}
We consider the function
\begin{equation}\label{3gdx}
g(x)=\frac{1}{e}\beta_3^x(h_3)=-\binom{x}{3}+\binom{x-1}{2}(\alpha+\beta+\gamma+1)-(x-2)(8\alpha+4\beta+2\gamma+1)+
(27\alpha+9\beta+3\gamma+1).
\end{equation}
By straightforward computations, from \eqref{3gdx} we get
\begin{equation}\label{3gdex}
g(x)=\alpha(\frac{1}{2}x^2-\frac{19}{2}x+44)+\beta(\frac{1}{2}x^2-\frac{11}{2}x+18)
+\gamma(\frac{1}{2}x^2-\frac{7}{2}x+8)+(-\frac{1}{6}x^3+x^2-\frac{17}{6}x+4).
\end{equation}
Note that, from \eqref{3gdx}, if $g(d)<0$ then $\qdepth(h_3)<d$. In particular, from \eqref{3gdex}, we can deduce
assertions like: If $g(11)=18\beta+30\gamma-128<0$ then $\qdepth(h_3)<10$.
\end{obs}

\subsection*{Aknowledgments} 

The second author, Mircea Cimpoea\c s, was supported by a grant of the Ministry of Research, Innovation and Digitization, CNCS - UEFISCDI, 
project number PN-III-P1-1.1-TE-2021-1633, within PNCDI III.

\end{document}